\documentclass[11pt]{amsart}
\usepackage{latexsym}
\usepackage{amscd}
\usepackage{amssymb}
\usepackage[all, cmtip]{xy}
\usepackage{amsmath}
\DeclareMathAlphabet{\eusm}{OT1}{eusm}{m}{n}
\newtheorem{theorem}{Theorem}[section]

\newtheorem{defi}[theorem]{Definition}
\newtheorem{cor}[theorem]{Corollary}
\newtheorem{lem}[theorem]{Lemma}

\newtheorem{remark}[theorem]{Remark}
\newtheorem{question}[theorem]{Question}

\def\Soc{\mbox{Soc\/}}

\def\End{\mbox{End\/}}

\begin{document}
\title{MacWilliams extending conditions and quasi-Frobenius rings}
\author{Pedro A. Guil Asensio}
\thanks{The work of the first author is partially supported by the Spanish Government under
		grant PID2020-113206GB-I00/AEI/10.13039/501100011033 which includes FEDER funds of the EU, and by Fundaci\'on S\'eneca of
		Murcia under grant
		19880/GERM/15.
	The work of the second author is partially supported by a grant from Simons Foundation (grant number 426367).}
\address{Departamento de Mathematicas, Universidad de Murcia, Murcia, 30100, Spain}
\email{paguil@um.es}
\author{Ashish K. Srivastava}
\address{Department of Mathematics and Statistics, Saint Louis University, St.
Louis, MO-63103, USA} \email{ashish.srivastava@slu.edu}

\maketitle

\begin{abstract}
MacWilliams proved that every finite field has the extension property for Hamming weight which was later extended in a seminal work by Wood who characterized finite Frobenius rings as precisely those rings which satisfy the MacWilliams extension property. In this paper, the question of when is a MacWilliams ring quasi-Frobenius is addressed. It is proved that a right or left noetherian left 1-MacWilliams ring is quasi-Frobenius thus answering the different questions asked in \cite{Mio,SZ}. We also prove that a right perfect, left automorphism-invariant ring is left self-injective. In particular, this yields that if $R$ is a right (or left) artinian, left automorphism-invariant ring, then $R$ is quasi-Frobenius, thus answering a question asked in \cite{Mio}. 
\end{abstract}

\bigskip

\bigskip

\section{Introduction}

\noindent A linear code of length $n$ over a ring $R$ is simply a submodule $L$ of $R^n$. The Hamming weight $wt(x)$ of an element $x=(x_1, \ldots, x_n)\in R^n$ is defined as the number of non-zero entries $x_i$ in $x$. In 1962, F. J. MacWilliams proved in her Ph.D. dissertation that any Hamming weight isometry between two codes over a finite field $\mathbb F$ extends to a monomial transformation of the ambient space $\mathbb F^n$ \cite{Mac}. Recall that a monomial transformation of $R^n$ is a map $f:R^n\rightarrow R^n$ of the form $f(x_1, \ldots, x_n)=(x_{\sigma(1)}u_1, \ldots, x_{\sigma(n)}u_n)$ for some permutation $\sigma \in S_n$ and invertible elements $u_i\in R$. As it was observed by Goldberg \cite{G}, the MacWilliams Extension Theorem is an analogue of the Witt’s Extension Theorem for quadratic forms (see e.g \cite{Lam}). The original proof of MacWilliams was later simplified by Bogart, Goldberg and Gordon \cite{BGG}. Ward and Wood gave an alternative proof of the MacWilliams Extension Theorem using character-theoretic approach \cite{WW}. 

A fundamental result in the study of linear codes over rings is the characterization of finite Frobenius rings by Wood \cite{Wood} as precisely those rings $R$ which satisfy the following MacWilliams extension property: every Hamming weight preserving isomorphism between left submodules of $R^n$ extends to a monomial transformation of $R^n$. Inspired by this, a ring $R$ is called {\it left MacWilliams} if every Hamming weight preserving homomorphism $f:L\rightarrow R^n$ from a left submodule $L$ of $R^n$ extends to an automorphism of $R^n$. Thus, if $R$ is left MacWilliams, then for $n=1$ the extension property tell us that each monomorphism $f:I\rightarrow R$ from a left ideal $I$ of $R$ extends to an automorphism of $_RR$. This is how automorphism-invariant modules naturally appear in the theory of linear codes over rings. 

Recall that a module $M$ which is invariant under automorphisms of its injective envelope is called an {\it automorphism-invariant module}. The study of modules which are invariant under the action of certain subsets of the endomorphism ring of their
injective envelope goes back to the pioneering work of Johnson and Wong \cite{JW} in which they characterized modules $M$ for which every homomorphism $f:L\rightarrow M$ from a submodule $L$ of $M$ extends to an endomorphism of $M$ as precisely those modules $M$ which are invariant under any endomorphism of their injective envelope $E(M)$. Such modules are called quasi-injective modules. Injective modules are trivially automorphism-invariant. Moreover, every quasi-injective module is also automorphism-invariant. It was shown in \cite{ESS} that a module $M$ is automorphism-invariant if and only if every monomorphism from a submodule of $M$ extends to an endomorphism of $M$. Automorphism-invariant modules are also known as pseudo-injective modules in the literature. This class of modules was first studied by Dickson and Fuller in \cite{DF} and it has been studied extensively in recent years. See \cite{Book}, \cite{GA-S}, \cite{GA-KT-S}, \cite{GQS} for more details on automorphism-invariant modules. If $R$ is a ring such that $_RR$ is a (quasi-)injective module then $R$ is called a left self-injective ring, whereas if $_RR$ is an automorphism-invariant module then $R$ is called a left automorphism-invariant ring.  

Clearly, the MacWilliams extension property for $n=1$ implies that $R$ is a left automorphism-invariant ring, and for $n=2$, the MacWilliams extension property implies that $R$ is left self-injective (see Lemma \ref{basic}).

A ring $R$ is called {\it quasi-Frobenius} if $R$ is left and right artinian and left and right self-injective. The class of quasi-Frobenius rings was introduced by Nakayama in his seminal paper \cite{N1} to study the duality between the categories of left and right finitely generated modules over them. Later, Nakayama \cite{N2} and Ikeda \cite{Ikeda} characterized quasi-Frobenius rings as those rings which are right or left artinian and right or left self-injective. Recall that a ring $R$ is called a {\it left perfect ring} if every left $R$-module has a projective cover. From Osofsky \cite{Osofsky} and Kato \cite{Kato}, it follows that if $R$ is a left (or right) perfect, left and right self-injective ring, then $R$ is quasi-Frobenius. In 1976, Carl Faith \cite{Faith1} asked whether every left or right perfect, left self-injective ring $R$ is quasi-Frobenius. Faith initially conjectured the answer to be ``no" even for a semiprimary ring but later changed his mind and conjectured the answer to be ``yes". So, now the following is known as Faith's Conjecture:

\medskip

{\bf Faith's Conjecture:} Every left or right perfect, left self-injective ring is quasi-Frobenius. 

\medskip

\noindent This conjecture still remains open. In this paper, we look at a variation of this conjecture for left automorphism-invariant rings and prove that if $R$ is a right perfect, left automorphism-invariant ring then $R$ is left self-injective. Now, if Faith's Conjecture is true then we will have that every right perfect, left automorphism-invariant ring is quasi-Frobenius. 

One of the main objectives of this paper is to answer the different questions raised in \cite{SZ} and \cite{Mio}. Namely, the following questions are asked in those papers:

\begin{question} \cite[Section 4, Question]{SZ}
Suppose $R$ is a right artinian ring which is left MacWilliams. Does it follow that $R$ is a quasi-Frobenius ring?
\end{question}

 \begin{question}\cite[Question~2.2]{Mio}.
	What can be said about left (or left-right) MacWilliams rings in general?
	Is a Noetherian / Noetherian commutative left MacWilliams ring necessarily Artinian (and
	hence, QF)
\end{question}

\noindent We define the notion of left $n$-MacWilliams ring and, in this terminology, a ring $R$ is left MacWilliams if it is left $n$-MacWilliams for each $n\ge 1$. We prove that any left $2$-MacWilliams ring is left self-injective. And that, indeed, a right or left noetherian (in particular, any left or right artinian) left 1-MacWilliams ring is quasi-Frobenius, thus giving a positive answer to the questions above. In the process, we note that a ring is left $1$-MacWilliams if and only if it is left automorphism-invariant, which allows us to also answer
the following question asked in \cite{Mio}:

\begin{question}\cite[Question~2.6]{Mio}.
	If $R$ is a right artinian, left automorphism-invariant ring, does it follow that $R$ is a quasi-Frobenius ring?
\end{question} 	

\noindent We refer to \cite{AF, Faith1, Good} for any undefined notion used in the text.
 
\section{Results}

\noindent We begin by first defining the terminology we will be using and making some basic observations. 

\begin{defi}
We will call a ring $R$ to be a left $n$-MacWilliams ring if every Hamming weight preserving homomorphism $f:L\rightarrow R^n$ from a left submodule $L$ of $R^n$ extends to an automorphism of $R^n$.
\end{defi}

\noindent Thus, a ring $R$ is left MacWilliams if it is $n$-MacWilliams for each $n\ge 1$. 

\begin{lem} \label{basic} Let $R$ be a ring. 
\begin{enumerate}
\item If $R$ is left 1-MacWilliams, then it is left automorphism invariant.
\item If $R$ is a left 2-MacWilliams, then it is left self-injective. 
\end{enumerate}
\end{lem}

\begin{proof}
(1) This follows easily from the definition as observed in introduction. 

(2) If $R$ is left 2-MacWilliams then this means $R^2$ is automorphism-invariant. Now, by \cite[Corollary 4.6]{Book}, $_RR$ is $_RR$-injective. Thus, $R$ is left self-injective.
\end{proof} 

\noindent Now, we would like to study when is a 1-MacWilliams ring left self-injective ring. 
The following technical lemma will be needed to prove our main result. Recall that, given a left module $M$, the socle of $M$, denoted by $\Soc(M)$, is the sum of all simple submodules of $M$. We will denote by $J=J(R)$ the Jacobson radical of a ring $R$.

\begin{lem}\label{finite}
	Let $R$ be a right perfect, left automorphism-invariant ring. Then $\Soc(_RR)$ is finitely generated and essential in $_RR$.  
\end{lem}	

\begin{proof}
	It is well known that the left socle of a right perfect ring is an essential left ideal (see \cite{Bass}), so we only need to show that $\Soc(_RR)$ is finitely generated. Assume on the contrary that this is not the case. As $R$ is, in particular, semiperfect, there is only a finite number of isomorphism classes of simple left modules. So there must exist an infinite direct sum, say $\oplus_{n\in\mathbb N} C_n$, of simple left modules isomorphic among them contained in $R$.  
	
	Let us call $u_m:C_m\rightarrow \oplus_{n\in\mathbb N}C_n$, for each $m\in \mathbb N$ and $v:\oplus_{n\in\mathbb N}C_n\rightarrow\, _RR$ the inclusions.
	And define, for each $m\in \mathbb N$, homomorphisms $f_m:\oplus_{n\in\mathbb N}C_n\rightarrow\, _RR$ as follows:
	\begin{itemize}
		\item $f_m|_{C_{3m}}=u_{3m+1}\circ \alpha_m$, where $\alpha_m:C_{3m}\rightarrow C_{3m+1}$ is an isomorphism.
		\item $f_m|_{C_{3m+1}}=u_{3m+2}\circ \beta_m$, where $\beta_m:C_{3m+1}\rightarrow C_{3m+2}$ is an isomorphism.
		\item $f_m|_{C_{3m+2}}=u_{3m}\circ \gamma_m$, where $\gamma_m:C_{3m+2}\rightarrow C_{3m}$ is an isomorphism.			
		\item $f_m|_{C_n}=u_n\circ 1_{C_n}$, if $n\neq 3m,3m+1,3m+2$.
	\end{itemize}	
	
	Note that $v\circ f_m:\oplus_{n\in\mathbb N}C_n\rightarrow\, _RR$ is a monomorphism and therefore, there exists a $g_m:_RR\rightarrow _RR$ such that $g_m\circ v= v\circ f_m$ since $R$ is left automorphism-invariant. Let $r_m=1-g_m(1)\in R$. This element $r_m$ has the following properties:
	\begin{enumerate}
		\item If $0\neq x\in C_{3m}$, then $x\cdot r_m=(1-g_m)(x)=x-u_{3m+1}\circ\alpha(x)\neq 0$ since $u_{3m+1}\circ\alpha(x)\in C_{3m+1}$. In particular, this means that $C_{3m}\nsubseteq l_R(r_m)$ and thus, $l_R(r_m)$ is not essential in $_RR$, where $l_R(r_m)$ denotes the left annihilator of $r_m$. Therefore, $r_m\notin J(R)$ since $J(R)$ equals the left singular ideal of $R$ (see e.g. \cite[Theorem~3.3]{Book}).
		\item If $0\neq x\in C_{3m}$, then $$x\cdot r_m^2=(1-g_{m})\circ (1-g_{m})(x)=(1-g_{m})(x-u_{3m+1}\circ\alpha(x))=$$
		$$x-u_{3m+1}\circ\alpha(x)-g_{m}(x-u_{3m+1}\circ\alpha(x))=$$
		$$x-u_{3m+1}\circ\alpha(x)-u_{3m+1}\circ\alpha(x)+u_{3m+2}\beta_m\circ\alpha_m(x)=$$
		$$=x+2u_{3m+1}\circ \alpha(x)-u_{3m+2}\circ\beta_m\circ\alpha_m(x)\neq 0$$
		again since $u_{3m+2}\circ\beta\circ\alpha(x)\in C_{3m+2}$, $u_{3m+1}\circ\alpha(x)\in C_{3m+1}$ and $x\in C_{3m}$. This means that $C_{3m}\nsubseteq l_R(r_m^2)$ and thus, $l_R(r_m^2)$ is not essential in $_RR$. In particular,  $r_m^2\notin J(R)$.
		\item If $x\in C_n$ with $n\neq 3m,3m+1$ or $3m+2$, then $g_m(x)=x$ and thus, $x\cdot r_m=(1-g_m)(x)=x-x=0$.
		\item $r_m\cdot r_l\in J(R)$ when $m\neq l$ since if $n\neq 3m,3m+1$ or $3m+2$, then $C_n\cdot r_m=0$ by (3). Whereas for $x\in C_{3m}\oplus C_{3m+1}\oplus C_{3m+2}$, we have that $x\cdot r_m\in C_{3m}\oplus C_{3m+1}\oplus C_{3m+2}$ and thus, $(x\cdot r_m)\cdot r_l=0$, again by (3).
		This means that $\oplus_{n\in \mathbb N}C_n\subseteq l_R(r_m\cdot r_l)$ and thus$,r_m\cdot r_l\in J(R)$. 
	\end{enumerate}
	Let us now consider the left ideal $L=\sum_{m\in\mathbb N}(Rr_m+J(R))/J(R)$ of 
	$R/J(R)$. As $R/J(R)$ is semisimple 
	(since $R$ is right perfect) $L$ must be finitely generated. So there exists an $m_0$ such that $L= \sum_{m\leq m_0}(Rr_m+J(R))/J(R)$. 
	And this means that there exist $s_m\in R$, for $m=1,\ldots,m_0$ and $j\in J(R)$ such that $r_{m_0+1}=\sum_{m\leq m_0}s_m\cdot r_m +j$. If we multiply on the right by $r_{m_0+1}$, then noting that  $r_{m_0+1}\cdot r_{m}\in J(R)$ for any $m\leq m_0$ in the view of property (4), we have that $r_{m_0+1}^2\in J(R)$, a contradiction to property (2). This shows that $\Soc(_RR)$ must be finitely generated.
\end{proof}	

We can now prove our main result.

\begin{theorem} \label{selfinj}
	Let $R$ be a right perfect left automorphism-invariant ring. Then $R$ is left self-injective. 
\end{theorem}

\begin{proof}
As $R$ is right perfect, it is in particular semiperfect. Thus, $_RR$ is a direct sum of indecomposable projective left ideals. Moreover, $\Soc(_RR)$ is finitely generated and essential in $_RR$ by the above lemma.
Call $J=J(R)$ the Jacobson radical of $R$ and let $\Lambda=\{P_1, \ldots, P_n\}$ be a representative set of the isomorphism classes of the indecomposable projective left ideals. As $R$ is semiperfect, each $P_i$ is the projective cover of the simple left module $C_i=P_i/JP_i$ and $\{C_1, \ldots, C_n\}$ is a representative set of the isomorphism classes of the simple left $R$-modules. We know from \cite[Theorem~4.23]{Book} that if $P_i$ is not quasi-injective, then $\End(P_i)/J(\End(P_i))\cong \mathbb{F}_2$. So we may order the set $\Lambda$ as follows:
\begin{enumerate}
	\item $P_i$ is a non-quasiinjective module with 
	$\End(P_i)/J(\End(P_i))\cong \mathbb{F}_2$ for $i=1,\ldots,k$.
	\item $P_i$ is a quasi-injective module for $i=k+1,\ldots,n$.
\end{enumerate} 

Let us study the two cases:

\begin{enumerate}

\item Let us first choose an $i=k+1,\ldots,n$. Note that, as $P_i$ is an indecomposable quasi-injective left ideal, its injective envelope $E(P_i)$ is also indecomposable and $\Soc(P_i)$ consists in a simple left module. Call $D_i$ this simple module.
We claim that $D_i$ cannot be isomorphic to any simple module in the socle of $P_{i'}$ for any $i'\in\{1,\ldots,n\}$ with $i'\neq i$. Assume on the contrary that $D_i$ is isomorphic to a simple module $D_{i'}$ in the socle of $P_{i'}$. Let $\varphi: D_i\rightarrow D_{i'}$ be an isomorphism. As $P_i\oplus P_{i'}$ is automorphism-invariant, $P_{i'}$ is $P_i$-injective and $P_{i'}$ is $P_i$-injective (see e.g. \cite[Corollary~4.6]{Book}). Therefore, $\varphi$ and $\varphi^{-1}$ extend to homomorphisms 
$f:P_i\rightarrow P_{i'}$ and
$g:P_{i'}\rightarrow P_i$ 
such that $g\circ f|_{D_i}=1_{D_i}$. 
As $D_i$ is essential in $P_i$ and $P_i$ is quasi-injective, we deduce that $g\circ f$ is an isomorphism and so, $P_i$ is isomorphic to a direct summand of $P_{i'}$. But this means that $P_i\cong P_{i'}$, since $P_{i'}$ is indecomposable. A contradiction that proves our claim.

\item If $i=1,\ldots,k$, $E(P_i)$ must be a direct sum of indecomposable modules $Q_j$ with $\End(Q_j)/J(\End(Q_j))\cong\mathbb{F}_2$,
by \cite[Corollary~4.29]{Book}. And each $Q_j$ contains a simple module by hypothesis. So $P_i$ contains  at least a simple module whose endomorphism ring must be isomorphic to $\mathbb{F}_2$.

\end{enumerate}

We wish to show that $k=0$ and therefore, $P_i$ is quasi-injective for every $i=1,\ldots,n$. Let us assume on the contrary that $k\geq 1$. 

\medskip

{\bf Claim:} There exists an $i_0\in \{1, \ldots, k\}$ and a simple module $D\subseteq \Soc(P_{i_0})$ which is not isomorphic to any simple module in $\Soc(P_i)$ for any $i\in\{1,\ldots,k\}$ with $i\neq i_0$. 

Assume this is not the case. Call $P=\oplus_{i=1}^k P_i$. We know that $E(P)$ is a direct sum of indecomposables, say $E(P)=\oplus Q_t$. Let us choose a $Q_t$. Then $Q_t$ must contain an essential simple module $D\subseteq \Soc(P_i)$ for some $i\in\{1\dots,k\}$, which must be isomorphic to some other simple in the socle of $Q_{t'}$ with $t'\neq t$ and thus, $Q_t\cong Q_{t'}$. But then, if $S=\End(E(P))$, we deduce that $S/J(S)=\sqcap M_{n_l}(K_l)$ is a product of matrix rings over division rings satisfying that each $n_l\geq 2$. And thus, $S/J(S)$ has no homomorphic image isomorphic to $\mathbb{F}_2$. This means that any element in $S/J(S)$ is the sum of two units (see e.g. \cite[Theorem~2.21]{Book}). And, as an element $s\in S$ is a unit if and only if $s+J(S)$ is a unit in $S/J(S)$, we deduce that any element in $S$ is the sum of two units. It follows that $P$ is invariant under any endomorphism of $E(P)$, since it is invariant under automorphisms of $E(P)$; what shows that $P$ is quasi-injective. This yields a contradiction, proving our claim that there exists an $i_0\in \{1, \ldots, k\}$ and a simple module $D\subseteq \Soc(P_{i_0})$ such that $D$ is not isomorphic to any simple module in $\Soc(P_i)$ for any $i=1,\ldots,k$ with $i\neq i_0$.

Without loss of generality, we may assume that $i=1$ and call $D_1$ that simple module. Let us now call $P'=\oplus_{i=2}^{k}P_i$. Again $P'$ is automorphism-invariant and $E(P')$ is a direct sum of indecomposables. The same argument of the claim shows that, if $P'$ is not quasi-injective, then there exists an $i_1\in \{2, \ldots, k\}$ such that $P_{i_1}$ contains a simple module which is not isomorphic to any simple module in the socle of $P_i$ for any $i\in \{2, \ldots, k\}$ with $i\neq i_1$. Again we may suppose that $i_1=2$ and call $D_2$ that simple module. Note that $D_2 \not\cong D_1$. 

Following this argument we find simple modules $D_i$ in the socle of $P_i$, for $i=1, \ldots, k$, such that $D_i$ is not isomorphic to any simple module in the socle of $P_{i'}$ for any $i'\in \{1,\ldots,k\}$ with $i'> i$.

On the other hand, if $i=k+1,\ldots,n$, then $P_i$ is quasi-injective and contains an essential simple module $D_i$ that cannot be isomorphic to $D_{i'}$ for any $i'\in\{1,\ldots,n\}$ with $i'\neq i$ as shown in Case~(1).
We have then proved that, for each $i=1,\ldots,n$, there exists a simple module in $\Soc(P_i)$ such that $D_i\ncong D_{i'}$ for any $i'\in\{1,\ldots,n\}$ with $i'\neq i$. And, as there are only $n$ isomorphism classes of simple left modules, this means that all simple modules in $\Soc(P_i)$ are isomorphic among them, for any $i=1,\ldots,n$.  But then, $P_i$ is quasi-injective for each $i=1,\ldots,n$ by \cite[Theorem~4.23]{Book}. A contradiction that proves our claim.

Therefore, $\oplus_{i=1}^nP_i$ is a direct sum of quasi-injective modules. But this means that $\oplus_{i=1}^nP_i$ is also quasi-injective by \cite[Corollary~4.7]{Book}.  Thus, as $_RR$ is a direct summand of a finite direct sum of copies of  $\oplus_{n=1}^t P_n$, it is also  quasi-injective, that is, $R$ is a left self-injective ring. 
\end{proof}

\noindent As a consequence, we have the following corollary that answers the Question 2.6 posed in \cite{Mio}. 

\begin{cor} \label{cor1}
Let $R$ be a right (or left) artinian ring. If $R$ is left automorphism-invariant, then $R$ is quasi-Frobenius. 
\end{cor}

In particular, our Corollary \ref{cor1} yields the following which answers \cite[Section 4, Question]{SZ} and improves \cite[Corollary 2.7]{Mio}. 

\begin{cor}
If $R$ is a right (or left) artinian, left 1-MacWilliams ring, then $R$ is quasi-Frobenius. 
\end{cor}  

\begin{remark}
In \cite[Question 2.3]{Mio}, it is asked whether every artinian automorphism-invariant module is a quasi-injective module. We would like to mention that an example of an indecomposable artinian automorphism-invariant module which is not quasi-injective is given in 
\cite[Example 3.1]{GQS} (see also \cite[Example~4.21]{Book}). 
\end{remark} 

We have implicitly shown in our proof of Theorem \ref{selfinj} that if $R$ is a right perfect, left automorphism-invariant ring, then each simple left $R$-module is isomorphic to a simple module in $\Soc(_RR)$. Thus, it follows that 

\begin{cor}
Every right perfect, left automorphism-invariant ring is a left pseudo-Frobenius ring. 
\end{cor}

\noindent Now, we proceed to extend our results to left or right noetherian, left automorphism-invariant rings. Recall that the left singular ideal $Z_l(R)$ of a ring $R$, is the left ideal consisting of all elements of $R$ whose left annihilator is an essential left ideal of $R$.

We can now give an affirmative answer to the question raised in \cite{Mio} of whether a left or right noetherian left 1-MacWilliams ring is quasi-Frobenius. 

\begin{theorem}\label{bounded}
	
	Let $R$ be a left or right noetherian ring. If $R$ is left automorphism invariant, then $R$ is quasi-Frobenius.
\end{theorem}

\begin{proof}
As $R$ is left automorphism-invariant, $R$ is semiregular (see e.g. \cite[Theorem~3.3]{Book}). And, as $R$ is left or right noetherian, $R/J(R)$ is semisimple. Thus, $R$ is semiperfect. Moreover, $J=J(R)$ equals the left singular ideal $Z_l(R)$ (see e.g. \cite[Theorem~3.3]{Book}).

Now, we may apply \cite[Proposition~3.31]{Good} in case $R$ is left noetherian, or \cite[Corollary 1.3]{Johns} in case $R$ is right noetherian, to prove that $J(R)$ is nilpotent. Therefore $R$ is right perfect and we can apply Theorem~\ref{selfinj} to show that it is left self-injective. Finally, $R$ is quasi-Frobenius, since it is also left or right noetherian.

\end{proof}

As every left 1-MacWilliams ring is left automorphism-invariant, we deduce:

\begin{cor}
	Every left 1-MacWilliams left or right noetherian ring is quasi-Frobenius. 
\end{cor}

\end{document}